\documentclass[journal,twoside,web]{ieeecolor}
\usepackage{lcsys} % For IEEE-CSC letters requires: Icsys.sty
\usepackage{xcolor} % When using IEEEclc.els  and for TAC JOURNAL
\usepackage{titlesec} % Adjust section spacing
\titlespacing*{\section}{0pt}{*0.5}{*0.5}
\titlespacing*{\subsection}{0pt}{*0.4}{*0.4}
\AtBeginDocument{%
	\setlength{\textfloatsep}{1pt plus 1pt minus 2pt}%  % between top/bottom float and body text
	\setlength{\floatsep}{2pt plus 1pt minus 1pt}%      % between consecutive floats on a float page
	\setlength{\intextsep}{2pt plus 1pt minus 1pt}%     % around h-placed (in-text) floats
	\setlength{\abovecaptionskip}{2pt}%                 % between float body and its caption
	\setlength{\belowcaptionskip}{2pt}%                 % below the caption
}

\useRomanappendicesfalse

\usepackage{amsfonts,mathtools}
\usepackage{graphicx}
\usepackage{xcolor}
\usepackage{bm}
\usepackage{booktabs,multirow}
\usepackage{cite}
\usepackage[colorlinks=true,linkcolor=blue,citecolor=blue,urlcolor=blue]{hyperref}

\newcommand{\titlesc}[1]{\title{\huge \vspace{0.7cm} \LARGE \centering {\textsc{{#1}}}}}

\definecolor{bittersweet}{rgb}{1.0, 0.44, 0.37}
\definecolor{blue(munsell)}{rgb}{0.0, 0.5, 0.69}
\definecolor{viridian}{rgb}{0.25, 0.51, 0.43}

\makeatletter
\def\tagform@#1{\maketag@@@{\textcolor{black}{(\ignorespaces#1\unskip\@@italiccorr)}}}
\makeatother

\usepackage[ruled,vlined,linesnumbered]{algorithm2e}
\SetKw{KwBreak}{break}
\usepackage{tikz}
\usetikzlibrary{arrows.meta,calc,positioning}
\usepackage{pgfplots}
\pgfplotsset{compat=1.18}

\newtheorem{problem}{Problem}
\newtheorem{example}{Example}
\newtheorem{theorem}{Theorem}
\newtheorem{corollary}{Corollary}
\newtheorem{remark}{Remark}

\let\labelindent\relax
\usepackage{enumitem}
\setlist[itemize]{leftmargin=*}

\makeatletter
\def\@begintheorem#1#2{%
	\@IEEEtmpitemindent\itemindent\topsep 0pt\rmfamily\trivlist
	\item[\hskip\labelsep{\noindent\bfseries #1~#2.}]\itemindent\@IEEEtmpitemindent}
\def\@opargbegintheorem#1#2#3{%
	\@IEEEtmpitemindent\itemindent\topsep 0pt\rmfamily\trivlist
	\item[\hskip\labelsep{\noindent\bfseries #1~#2\normalfont~(#3)\bfseries.}]\itemindent\@IEEEtmpitemindent}
\renewenvironment{proof}[1][Proof]{%
	\par\normalfont
	\@IEEEtmpitemindent\itemindent\topsep 0pt\trivlist
	\item[\hskip\labelsep{\noindent\bfseries #1.}]\itemindent\@IEEEtmpitemindent\ignorespaces
}{%
	\unskip\nobreak\hfill\mbox{\rule[0pt]{1.3ex}{1.3ex}}\par
	\endtrivlist\@endpefalse
}
\makeatother	

\begin{document}
	\titlesc{Revisiting the PBH Test: \\Fast Uncontrollability Certificates via Krylov Methods}
	
	\author{Ahmad F. Taha$^{\diamond}$, Mohamad H. Kazma, and Abdallah A. Albustami\vspace{-1cm}
		\thanks{
			$^\diamond$Corresponding author. All authors are with the Civil \& Environmental Engineering Department at Vanderbilt University. A. Taha has a secondary appointment in ECE. Email: \url{ahmad.taha@vanderbilt.edu}.    }
	}
	\maketitle
	\thispagestyle{empty}
	\pagestyle{empty}
	
	\begin{abstract}
		This letter revisits the classical PBH test through the lens of finite-horizon reachability. By casting state transfer as a minimum energy, primal optimization problem, we show that unreachable state-space maneuvers admit dual infeasibility certificates. These certificates are computable without forming the controllability matrix meaning that uncontrollability can be efficiently certified. We prove that any such certificate is a linear combination of uncontrollable generalized eigenvectors, thereby providing a spectral interpretation without a global eigendecomposition. We also devise algorithms based on Krylov subspace methods that extract some of the uncontrollable PBH modes from a certificate and demonstrate favorable scaling on large dynamic networks with thousands of nodes. 
	\end{abstract}
	% \vspace{-0.2cm}
	
	% \vspace{-0.15cm}
	
	\section{Introduction and Letter Contributions}\label{sec:intro}
	\IEEEPARstart{C}{ontrollability} characterizes whether an input can steer the state of a linear time-invariant (LTI) system throughout its state-space. For discrete-time dynamics $\mathbf{x}_{k+1}=\mathbf{A}\mathbf{x}_k+\mathbf{B}\mathbf{u}_k$, the classical rank condition $\mathrm{rank}\,\boldsymbol{\mathcal{C}}_n=n$ (with $\boldsymbol{\mathcal{C}}_n=[\mathbf{B},\mathbf{A}\mathbf{B},\dots,\mathbf{A}^{n-1}\mathbf{B}]$) produces a binary algebraic test. The Popov-Belevitch-Hautus (PBH) test complements this view by providing a modal characterization: pair $(\mathbf{A},\mathbf{B})$ is controllable if and only if $\mathrm{rank}\,[\lambda \mathbf{I}-\mathbf{A}\;\;\mathbf{B}]=n$ for every eigenvalue $\lambda$ of $\mathbf{A}$. Equivalently, there exists no nonzero left eigenvector $\mathbf{w}$ satisfying $\mathbf{w}^\top\mathbf{A}=\lambda\,\mathbf{w}^\top$ and $\mathbf{w}^\top\mathbf{B}=\mathbf{0}$~\cite{hautus1969controllability}. Beyond deciding controllability, the PBH test localizes \emph{how} controllability fails by identifying the uncontrollable eigenvalues, making it a diagnostic tool for analysis. An uncontrollable eigenvalue $\lambda_i$ of $\mathbf{A}$ is one whose associated left eigenvector $\mathbf{w}_i$ satisfies $\mathbf{w}_i^\top\mathbf{B}=\mathbf{0}$, so that the input $\mathbf{B}$ has no effect on that mode. While these criteria are equivalent in exact arithmetic, their numerical evaluation is delicate and depends on the conditioning of the underlying spectral computations.
	
	Scaling controllability tests to modern networked systems remains nontrivial. Power grids, transportation, and water networks (most are \textit{Kalman-uncontrollable})	routinely produce sparse state-space matrices with $n$ in the thousands, where dense eigendecomposition is impractical---and where left-eigenvector computations can be ill-conditioned in the presence of weakly damped modes. %clustered spectra, 
	In large-scale systems, Krylov subspace methods can extract only selected eigen-information using matrix vector products rather than dense factorization~\cite{saad2011numerical,sorensen1992implicit}. Related lines in network controllability leverage PBH logic through eigenvalue multiplicities and sparse spectral computations~\cite{yuan2013exact}. These approaches scale, but they still require computing and validating invariant subspaces over spectral regions, and they remain \emph{mode-centric} even when only a specific maneuver is of interest.
	
	\vspace{0.15cm}
	
	\noindent \textsc{\textbf{Letter Objective and Approach}.} The letter's objective is as follows: certifying uncontrollability of discrete-time (DT) systems via a meaningful and scalable certificate.  Specifically, this letter develops a complementary, \emph{target-centric} route. We start from finite-horizon reachability and a minimum energy optimization problem formulation, which asks for the smallest-energy input sequence that realizes a desired displacement $\Delta \mathbf{x}$ over $N$ steps. When $\Delta \mathbf{x}$ lies outside the reachable subspace, the problem becomes primal infeasible. In optimization, \textit{Theorems of Alternatives} certify such infeasibility by a separating vector~\cite[Chapter 5.8]{boyd2004convex}. Modern solvers return this object as a \emph{dual ray} as a standard component of their output. The literature commonly uses certificates to decide feasibility, but has rarely exploited them as objects for spectral analysis. Our main observation is that infeasibility certificates carry additional information beyond a proof of infeasibility: they encode PBH-relevant eigenstructure. In particular, certificates decompose into uncontrollable generalized eigenspaces. This establishes an explicit \textit{spectral geometric connection}---an infeasibility certificate is a structured linear combination of uncontrollable eigenvectors.
    
		\noindent \textsc{\textbf{Contributions}.} The letter demonstrates that this certificate-generation route is $\boldsymbol{\mathcal{C}}_N$-building free and \textcolor{black}{implementable without a global eigendecomposition}: \textit{(i)} infeasibility detection and certificate extraction use iterative products with $\mathbf{A}$, $\mathbf{B}$, and their transposes; and \textit{(ii)} modal extraction operates on a low-dimensional subspace generated by $\mathbf{y}$, reducing spectral computation to a companion matrix of degree $r\ll n$. \textcolor{black}{The methods then complement the PBH test by producing terminal target-centric test of uncontrollability via scalable certificates.} The letter's \textit{contributions} are:
	\begin{itemize}
		\item establishing a certificate-based reinterpretation of PBH by proving that $\mathrm{Null}(\boldsymbol{\mathcal{C}}_N^\top)$ is composed of uncontrollable generalized left eigenvectors and that any infeasibility certificate is a strict linear combination of such modes (Sec.~\ref{sec:pbh_cert});
		\item developing a \emph{target-aware} diagnostic: certificates satisfying $\mathbf{y}^\top\Delta \mathbf{x}\neq 0$ isolate only the modes that obstruct a desired maneuver (Sec.~\ref{sec:pbh_cert});
		\item presenting a scalable framework that computes certificates without forming $\boldsymbol{\mathcal{C}}_N$ and extracts obstructing eigenvalues/modes from certificates using shifted inverse iteration and Krylov-based reduction (Sec.~\ref{sec:pbh_cert} and~\ref{sec:extraction});
		\item reporting evidence comparing runtime and accuracy against PBH, Gramian, and least squares implementations (Sec.~\ref{sec:case}).
	\end{itemize}

	\section{Problem Formulation and Background}\label{sec:prob}
Consider the discrete-time LTI system
	\begin{equation}\label{equ:sys}
		\mathbf{x}_{k+1}=\mathbf{A}\mathbf{x}_k+\mathbf{B}\mathbf{u}_k,\qquad k=0,1,\dots,N-1,
	\end{equation}
	with state $\mathbf{x}_k\in\mathbb{R}^n$, input $\mathbf{u}_k\in\mathbb{R}^m$, and horizon $N\in\mathbb{N}$.  Given an initial condition $\mathbf{x}_0$ and a desired terminal state $\mathbf{x}_N$, we recall a standard reachability identity that supports the following viewpoint adopted in this letter: \emph{uncontrollability can be detected and explained via certificates, not only via spectral tests.}  To drive the DT LTI system~\eqref{equ:sys} from an initial state $\mathbf{x}_0$ to a target $\mathbf{x}_N$ in $N$ steps, we expand the dynamics accordingly:
	\begin{equation*}
		\underbrace{\mathbf{x}_N - \mathbf{A}^N \mathbf{x}_0}_{\Delta \mathbf{x}} = \boldsymbol{\mathcal{C}}_N \underbrace{\begin{bmatrix}\mathbf{u}_{N-1}^\top & \cdots & \mathbf{u}_0^\top\end{bmatrix}^\top}_{\mathbf{u}}.
	\end{equation*}
	Defining the target displacement as $\Delta \mathbf{x} = \mathbf{x}_N - \mathbf{A}^N \mathbf{x}_0$, we obtain the fundamental linear constraint $\boldsymbol{\mathcal{C}}_N \mathbf{u} = \Delta \mathbf{x}$. {
		For \(N\ge n\), the Cayley-Hamilton theorem implies that \(\operatorname{range}(\boldsymbol{\mathcal{C}}_N)=\operatorname{range}(\boldsymbol{\mathcal{C}}_n)\), so horizons beyond \(n\) do not generate new reachable directions. Hence, for feasibility testing, \(N=n\) is sufficient without loss of generality, although we retain \(N\ge n\) to allow an arbitrary prescribed horizon. Also, since \(\boldsymbol{\mathcal{C}}_N\in\mathbb{R}^{n\times mN}\), the minimum-energy problem is underdetermined only when \(mN>n\), or more precisely when \(mN>\operatorname{rank}(\boldsymbol{\mathcal{C}}_N)\).} We resolve this redundancy by solving the \textit{minimum energy control problem} (MECP):
	\begin{equation}~\label{equ:mecp_alt}
		\begin{aligned}
			\hspace{-0.35cm}\mathrm{\textbf{MECP}}: \;\; & \underset{\mathbf{u}}{\text{min}} 
			& & J(\mathbf{u}) = \frac{1}{2} \|\mathbf{u}\|_2^2 \;\;\; \text{s.t.} \;\;\; \boldsymbol{\mathcal{C}}_N \mathbf{u} = \Delta \mathbf{x}.
		\end{aligned}
	\end{equation}

	To find the solution to MECP, $\boldsymbol{\mathcal{C}}_N$ has to have full row rank. \textcolor{black}{In large-scale networks, the system is often \emph{structurally} rank-deficient ($\mathrm{rank}(\boldsymbol{\mathcal{C}}_n)<n$, a property of the pair $(\mathbf{A},\mathbf{B})$ alone), in which case $\Delta\mathbf{x}\notin\mathrm{range}(\boldsymbol{\mathcal{C}}_n)$ for any target not in the reachable subspace, rendering the MECP infeasible for such targets.} We interpret this uncontrollability-induced infeasibility through primal dual infeasibility certificates from convex optimization.

	% \subsection{Dual Infeasibility and the Theorem of Alternatives}
	Infeasibility in optimization is typically detected via dual forms of the main primal problem. Many solvers present a certificate vector as the output when the problem is infeasible. The \textit{Fredholm Alternative}~\cite{Fredholm1903}, a specific instance of \textit{Farkas' Lemma}~\cite[Chapter 5.8.3]{boyd2004convex}, when applied to~\eqref{equ:mecp_alt} states that $\boldsymbol{\mathcal{C}}_N \mathbf{u} = \Delta \mathbf{x}$ is infeasible if and only if there exists a vector $\mathbf{y}$ such that:
	\begin{equation}\label{equ:certificate_alt}
		\boxed{\boldsymbol{\mathcal{C}}_N^{\top} \mathbf{y} = \mathbf{0},\qquad \mathbf{y}^{\top}\Delta \mathbf{x}\neq 0.}
	\end{equation}
	Such a vector $\mathbf{y}$ is called a \textit{certificate of infeasibility}~\cite{boyd2004convex}, also referred to as a \textit{certificate of inconsistency}, as it proves that $\Delta \mathbf{x}$ is not in the range of $\boldsymbol{\mathcal{C}}_N$. It represents an unreachable direction in the state-space: no input sequence can produce a component of the state along $\mathbf{y}$ (since $\boldsymbol{\mathcal{C}}_N^\top \mathbf{y} = 0$), yet the desired target requires displacement in that direction. Optimization solvers return this dual vector $\mathbf{y}$ when declaring infeasibility, \textcolor{black}{but it is possible to compute these certificates without optimization; the objective pursued herein}. The problem formulation studied in this letter is given next. 
	
	\begin{problem}[Fast Uncontrollability Certification]\label{pbm:pbm1} Given $(\mathbf{A},\mathbf{B})$, a horizon $N$, and a target displacement $\Delta \mathbf{x}$, decide whether $\Delta \mathbf{x}\in\mathrm{range}(\boldsymbol{\mathcal{C}}_N)$.
		If $\Delta \mathbf{x}$ is unreachable, compute an infeasibility certificate $\mathbf{y}$ satisfying \eqref{equ:certificate_alt} without forming $\boldsymbol{\mathcal{C}}_N$ and use $\mathbf{y}$ to identify the $r$ uncontrollable PBH modes of $(\mathbf{A},\mathbf{B})$.
	\end{problem}
	
	The certificate of infeasibility $\mathbf{y}$ is related to the classical PBH rank test which is an equivalent test for controllability. 
	%  a replacement for rank of the intractable $\boldsymbol{\mathcal{C}}_N$. 
	The PBH test states that a system is uncontrollable if and only if there exists a left eigenvector $\mathbf{w}^\top$ of $\mathbf{A}$ (associated with eigenvalue $\lambda$) such that $\mathbf{w}^\top \mathbf{B} = \mathbf{0}_m$ \cite{kailath1980linear}. An eigenvector $\mathbf{w}$ can be a valid certificate of infeasibility---or the system being uncontrollable. One could show that the PBH test fails if and only if there exists an infeasibility certificate $\mathbf{y}$ satisfying~\eqref{equ:certificate_alt}. We start with two examples showing how the certificates of infeasibility can be generated from the PBH test.
	\begin{example}[Coinciding Certificates]
		Consider a diagonal system with three decoupled modes, where $\mathbf{x}_3$ is unactuated:
		% \begin{equation*}
			$\mathbf{A} = \mathrm{diag}(1, 2, 3),  
			\mathbf{B} = [1, 1, 0]^\top, 
			\Delta \mathbf{x} = [1, 1, 1]^\top.$
			% \end{equation*}
		The controllability matrix $\boldsymbol{\mathcal{C}}_3$ has a zero third row. The PBH test identifies $\mathbf{w}_3 = [0, 0, 1]^\top$ as the unique uncontrollable mode ($\mathbf{w}_3^\top \mathbf{B} = 0$). Solving MECP~\eqref{equ:mecp_alt}, a solver returns a dual certificate $\mathbf{y}$ satisfying $\boldsymbol{\mathcal{C}}_3^\top \mathbf{y} = 0$ and $\mathbf{y}^\top \Delta \mathbf{x} > 0$. By inspection, the null space of $\boldsymbol{\mathcal{C}}_3^{\top}$ is spanned by $\mathbf{y} = [0, 0, 1]^\top$. In this clean case, the optimization certificate $\mathbf{y}$ is \textit{identical} to the PBH eigenvector $\mathbf{w}_3$.
		
	\end{example}
	
	\begin{example}[Mixed Certificates]
		Consider now a similar system where the input actuates only the first of three modes:
		% \begin{equation*}
			$\mathbf{A} = \mathrm{diag}(1, 2, 3), \quad 
			\mathbf{B} = [1, 0, 0]^\top, \quad
			\Delta \mathbf{x} = [0, 1, 1]^\top.$
			Here, two modes are uncontrollable: $\lambda_2=2$ ($\mathbf{w}_2=\mathbf{e}_2$) and $\lambda_3=3$ ($\mathbf{w}_3=\mathbf{e}_3$). The target $\Delta \mathbf{x}$ requires movement in \textit{both} directions. \textcolor{black}{One natural certificate is the orthogonal projection of $\Delta\mathbf{x}$ onto the null space:} $\mathbf{y} = [0, 1, 1]^\top$.
			Observe that while $\mathbf{y}$ is a valid certificate ($\boldsymbol{\mathcal{C}}_3^\top \mathbf{y} = \mathbf{0}$), it is \textit{not} an eigenvector of $\mathbf{A}$, as $\mathbf{A}^\top \mathbf{y} = [0, 2, 3]^\top \neq \lambda \mathbf{y}$. The certificate $\mathbf{y} = \mathbf{w}_2 + \mathbf{w}_3$ is a linear combination of modes, and hence is \textit{not} unique. \end{example}

		\section{Spectral Decomposition of Certificates}\label{sec:pbh_cert}
		We now establish the link between the geometric certificate of infeasibility $\mathbf{y}$ and the dynamic modes of the system (identified by the PBH test). 	\textcolor{black}{The first result is  stated for general $\mathbf{A}$ (including non-diagonalizable systems with Jordan blocks), with the diagonalizable case as a corollary afterward.}

		\begin{theorem}[\textcolor{black}{Linking PBH Test and Infeasibility Certificates}]\label{thm:PBHcert}
			\textcolor{black}{Consider the LTI system~\eqref{equ:sys} with $\mathbf{A} \in \mathbb{R}^{n \times n}$. The MECP~\eqref{equ:mecp_alt} is {infeasible} if and only if there exists a certificate vector $\mathbf{y} \in \mathbb{R}^n\setminus\{\mathbf{0}\}$ satisfying the following two conditions.}
			
			\noindent \textit{(C1)} \textcolor{black}{Vector $\mathbf{y}$ lies in the \emph{uncontrollable subspace} $\bm{\mathcal{V}}_{\mathrm{unc}} = \mathrm{Null}(\boldsymbol{\mathcal{C}}_N^\top)$, which is spanned by the generalized left eigenvectors of $\mathbf{A}$ associated with uncontrollable modes, that is,}
			\begin{align}
			\textcolor{black}{	\mathbf{y} \in \bm{\mathcal{V}}_{\mathrm{unc}} }
				&= \textcolor{black}{\mathrm{Null}(\boldsymbol{\mathcal{C}}_N^\top) = \mathrm{span}\big\{\mathbf{v} \mid \exists \lambda, k : (\mathbf{A}^\top - \lambda \mathbf{I})^k \mathbf{v} = 0} \nonumber \\
				&\qquad \textcolor{black}{\text{and } \mathbf{v}^\top \boldsymbol{\mathcal{C}}_N = 0 \big\}.} \label{eq:uncontrollable_subspace}
			\end{align}
			\textit{(C2)} \textcolor{black}{The target state satisfies $\mathbf{y}^\top \Delta \mathbf{x} \neq 0$.}
		\end{theorem}

		\begin{proof}
			By Cayley-Hamilton, $\mathrm{Null}(\boldsymbol{\mathcal{C}}_N^\top) = \mathrm{Null}(\boldsymbol{\mathcal{C}}_n^\top)$ for $N \ge n$; we therefore consider $N=n$ and write $\bm{\mathcal{Y}} := \mathrm{Null}(\boldsymbol{\mathcal{C}}_n^\top)$.
			
			\noindent\textit{(Necessity: infeasible $\Rightarrow$ $\exists\,\mathbf{y}$ satisfying C1 and C2.)} Since the MECP is infeasible, the Fredholm Alternative~\cite{Fredholm1903} guarantees the existence of $\mathbf{y} \neq \mathbf{0}$ with $\boldsymbol{\mathcal{C}}_n^\top \mathbf{y} = \mathbf{0}$ and $\mathbf{y}^\top \Delta\mathbf{x} \neq 0$. Condition (C2) is immediate. For (C1), we establish the structural identity $\bm{\mathcal{Y}} = \bm{\mathcal{V}}_{\mathrm{unc}}$ by showing both inclusions.
			
			\noindent Let $\mathbf{y} \in \bm{\mathcal{Y}}$ and set $\mathbf{z} = \mathbf{A}^\top \mathbf{y}$ and verify $\mathbf{z}^\top \mathbf{A}^k \mathbf{B} = \mathbf{0}^\top$ for $k = 0, \ldots, n-1$. \textit{For $k = 0, \ldots, n-2$} we have \ $\mathbf{z}^\top \mathbf{A}^k \mathbf{B} = \mathbf{y}^\top \mathbf{A}^{k+1}\mathbf{B} = \mathbf{0}^\top$, since $k+1 \le n-1$ and $\mathbf{y} \in \bm{\mathcal{Y}}$.  \textit{For $k = n-1$:}\ $\mathbf{z}^\top \mathbf{A}^{n-1}\mathbf{B} = \mathbf{y}^\top \mathbf{A}^n \mathbf{B}$. By the Cayley-Hamilton theorem, $\mathbf{A}^n = \sum_{j=0}^{n-1} c_j \mathbf{A}^j$, so
				$\mathbf{y}^\top \mathbf{A}^n \mathbf{B} = \sum_{j=0}^{n-1} c_j (\mathbf{y}^\top \mathbf{A}^j \mathbf{B}) = \mathbf{0}^\top$
				since $\mathbf{y} \in \bm{\mathcal{Y}}$ implies $\mathbf{y}^\top \mathbf{A}^j \mathbf{B} = \mathbf{0}^\top$ for every $j = 0, \ldots, n-1$, then $\mathbf{z} = \mathbf{A}^\top \mathbf{y} \in \bm{\mathcal{Y}}$, confirming that $\bm{\mathcal{Y}}$ is $\mathbf{A}^\top$-invariant.

			Now since $\bm{\mathcal{Y}}$ is a finite-dimensional $\mathbf{A}^\top$-invariant subspace, the restriction $\mathbf{A}^\top|_{\bm{\mathcal{Y}}}$ is a well-defined linear operator on $\bm{\mathcal{Y}}$. By the Jordan decomposition theorem applied to $\mathbf{A}^\top|_{\bm{\mathcal{Y}}}$,
			$\bm{\mathcal{Y}} = \bigoplus_j \bm{\mathcal{E}}_j,$
			where $\bigoplus$ defines the Minkowski sum and each $\bm{\mathcal{E}}_j = \ker\bigl((\mathbf{A}^\top - \lambda_j\mathbf{I})^{n_j}\bigr) \cap \bm{\mathcal{Y}}$ is the generalized eigenspace of $\mathbf{A}^\top|_{\bm{\mathcal{Y}}}$ associated with eigenvalue $\lambda_j$, and $n_j = \dim \bm{\mathcal{E}}_j$.
			Every $\mathbf{v} \in \bm{\mathcal{Y}}$ is therefore a linear combination of generalized eigenvectors of $\mathbf{A}^\top$ that lie in $\bm{\mathcal{Y}}$.
			
		 We now show that each $\lambda_j$ is an uncontrollable eigenvalue.			Let $\mathbf{w}_j \in \bm{\mathcal{E}}_j$ be a base eigenvector: $\mathbf{A}^\top \mathbf{w}_j = \lambda_j \mathbf{w}_j$, equivalently $\mathbf{w}_j^\top \mathbf{A} = \lambda_j \mathbf{w}_j^\top$. Since $\mathbf{w}_j \in \bm{\mathcal{Y}} = \mathrm{Null}(\boldsymbol{\mathcal{C}}_n^\top)$, the $k=0$ block gives $\mathbf{w}_j^\top \mathbf{B} = \mathbf{0}^\top$. By the PBH test, $\lambda_j$ is an uncontrollable eigenvalue of $(\mathbf{A},\mathbf{B})$. Then,	for any $\mathbf{v} \in \bm{\mathcal{E}}_j$, by construction $\bm{\mathcal{E}}_j \subseteq \bm{\mathcal{Y}}$, so $\mathbf{v} \in \bm{\mathcal{Y}}$ and $\mathbf{v}^\top \boldsymbol{\mathcal{C}}_n = \mathbf{0}$. Thus \emph{every} generalized eigenvector appearing in the Jordan chain of $\mathbf{A}^\top|_{\bm{\mathcal{Y}}}$ lies in $\bm{\mathcal{Y}}$ and is associated with an uncontrollable eigenvalue.
			
		The previous steps show that $\bm{\mathcal{Y}} \subseteq \bm{\mathcal{V}}_{\mathrm{unc}}$. For the reverse inclusion $(\supseteq)$, let $\mathbf{v} \in \bm{\mathcal{V}}_{\mathrm{unc}}$, i.e., $\mathbf{v}$ is a generalized eigenvector of $\mathbf{A}^\top$ associated with some uncontrollable eigenvalue $\lambda_j$. By definition~\eqref{eq:uncontrollable_subspace}, the base eigenvector $\mathbf{w}_j$ in the Jordan chain satisfies $\mathbf{w}_j^\top \mathbf{B} = \mathbf{0}^\top$ (the PBH condition), and it follows that $\mathbf{v}^\top \boldsymbol{\mathcal{C}}_n = \mathbf{0}$, hence $\mathbf{v} \in \mathrm{Null}(\boldsymbol{\mathcal{C}}_n^\top) = \bm{\mathcal{Y}}$. Therefore $\bm{\mathcal{Y}} = \bm{\mathcal{V}}_{\mathrm{unc}}$.
			
			\noindent\textit{(Sufficiency: $\mathbf{y}$ satisfies C1 and C2 $\Rightarrow$ infeasible.)}
			Suppose $\mathbf{y} \in \bm{\mathcal{V}}_{\mathrm{unc}} = \bm{\mathcal{Y}} = \mathrm{Null}(\boldsymbol{\mathcal{C}}_n^\top)$ (C1), so $\boldsymbol{\mathcal{C}}_n^\top \mathbf{y} = \mathbf{0}$, i.e., $\mathbf{y}^\top \mathbf{A}^k \mathbf{B} = \mathbf{0}^\top$ for $k = 0, \ldots, n-1$. Combined with $\mathbf{y}^\top \Delta\mathbf{x} \neq 0$ (C2), the Fredholm Alternative applied to $\boldsymbol{\mathcal{C}}_n \mathbf{u} = \Delta\mathbf{x}$ gives infeasibility. This concludes the proof.
		\end{proof}
		
		\normalcolor

		Theorem~\ref{thm:PBHcert} characterizes the uncontrollable subspace defined as $\bm{\mathcal{V}}_{\mathrm{unc}} = \mathrm{Null}(\boldsymbol{\mathcal{C}}_N^\top)$. Any valid certificate $\mathbf{y}$ is a nonzero element of $\bm{\mathcal{V}}_{\mathrm{unc}}$ additionally satisfying $\mathbf{y}^\top\Delta\mathbf{x}\neq 0$; the set of valid certificates $\{\mathbf{y} \in \bm{\mathcal{V}}_{\mathrm{unc}}\setminus\{\mathbf{0}\} : \mathbf{y}^\top\Delta\mathbf{x}\neq 0\}$ is \emph{not} itself a subspace. The existence of such a vector $\mathbf{y}$ alone does \textit{not} imply infeasibility for \textit{every} target; it only identifies a direction in which control is impossible.

		\begin{corollary}[Diagonalizable Case]\label{thm:PBHcertnond}\label{cor:interpretuncont}
			\textcolor{black}{Consider the LTI system~\eqref{equ:sys} where $\mathbf{A} \in \mathbb{R}^{n\times n}$ is diagonalizable with distinct eigenvalues. The MECP is infeasible if and only if there exists $\mathbf{y} = \sum_{j\in\bm{\mathcal{I}}}\alpha_j\mathbf{w}_j\neq\mathbf{0}$ (where $\bm{\mathcal{I}}=\{j\mid\mathbf{w}_j^\top\mathbf{B}=\mathbf{0}^\top\}$) with $\mathbf{y}^\top\Delta\mathbf{x}\neq 0$. In this case $\bm{\mathcal{V}}_{\mathrm{unc}}=\mathrm{span}\{\mathbf{w}_j:j\in\bm{\mathcal{I}}\}$.}
		\end{corollary}
		
		\begin{proof}
			\textcolor{black}{Theorem~\ref{thm:PBHcert} applied to the diagonalizable case: since $\{\mathbf{w}_1,\ldots,\mathbf{w}_n\}$ form a basis, we expand $\mathbf{y}=\sum_i\alpha_i\mathbf{w}_i$ and the orthogonality condition $\mathbf{y}^\top\mathbf{A}^k\mathbf{B}=\mathbf{0}^\top$ becomes $\sum_i\alpha_i\lambda_i^k(\mathbf{w}_i^\top\mathbf{B})=\mathbf{0}^\top$ for $k=0,\ldots,N-1$. Writing $\boldsymbol{\gamma}_i^{(\ell)}=\alpha_i(\mathbf{w}_i^\top\mathbf{b}_\ell)\in\mathbb{R}$ for each column $\mathbf{b}_\ell$ of $\mathbf{B}$, the Vandermonde system $\mathbf{V}\boldsymbol{\gamma}^{(\ell)}=\mathbf{0}$ (with $V_{ki}=\lambda_i^k$) has the unique solution $\boldsymbol{\gamma}^{(\ell)}=\mathbf{0}$ since distinct eigenvalues make $\mathbf{V}$ have full column rank. Hence $\alpha_i\neq 0\Rightarrow\mathbf{w}_i^\top\mathbf{b}_\ell=0$ for all $\ell$, i.e., $\mathbf{w}_i^\top\mathbf{B}=\mathbf{0}^\top$.}
		\end{proof}
	
		\vspace{0.15cm}
		\noindent \textsc{\textbf{Implementation of Infeasibility Certificates}.} The PBH test is valuable for modal diagnosis, but it typically requires eigenvalue computations: for dense $n\times n$ matrices this has $\bm{\mathcal{O}}(n^3)$ cost, and even in sparse settings it relies on iterative spectral routines whose convergence can depend strongly on eigenvalue separation (spectral gaps). In large-scale systems, it can be advantageous to compute an {uncontrollability certificate} directly. 
		{We propose Alg.~\ref{alg:inviter} as a generic numerical procedure to identify uncontrollability certificates. This algorithm handles both diagonalizable and non-diagonalizable systems by verifying spectral properties alongside the geometric requirements of the Fredholm Alternative.} 
		
First, recall that infeasibility of $\boldsymbol{\mathcal{C}}_N \mathbf{u}=\Delta \mathbf{x}$ is certified by a nonzero $\mathbf{y}$ with $\boldsymbol{\mathcal{C}}_N^\top \mathbf{y}=\mathbf{0}$ and $\mathbf{y}^\top\Delta \mathbf{x}\neq 0$. This certificate can be computed \emph{without forming} $\boldsymbol{\mathcal{C}}_N$, using matrix-free products with $\mathbf{A}^\top$ and $\mathbf{B}^\top$ which is precisely the structure exploited by Krylov subspace methods~\cite{saad2011numerical}. In this letter we use the \emph{shifted inverse iteration}, which is well suited to this matrix-free setting. \textcolor{black}{Rather than assume a user-provided shift $\sigma$, the implementation first generates a small candidate set of shifts $\Sigma$ from Gershgorin-type spectral bounds for $\mathbf{A}$; see~\cite{varga2011gervsgorin}. The shift set $\Sigma$ is inexpensive to generate: the required Gershgorin bounds are obtained directly from the diagonal entries and row or column sums of $|\mathbf A|$, so this preprocessing scales linearly with the number of nonzeros. One then selects a shift $\sigma\in\Sigma$ and runs inverse iteration with that fixed value. The formal version of Alg.~\ref{alg:inviter} is a \emph{certificate-search} routine: it scans the shifts in $\Sigma$ and returns the first successful pair $(\sigma^\star,\mathbf y_{\sigma^\star})$.} For a fixed shift $\sigma\in\Sigma$, repeatedly applying $(\mathbf{A}^\top-\sigma \mathbf{I})^{-1}$ and normalizing drives the iterate toward a left eigenvector associated with the eigenvalue closest to $\sigma$; see~\cite{ipsen1997inverse,saad2011numerical}. Thus $\Sigma$ is a small set of candidate shifts generated at the beginning, and the role of the scan is to provide coarse spectral coverage rather than an exact eigenvalue estimate. \textcolor{black}{In the forensic variant all shifts are scanned and successful certificates retained.} Each iteration is dominated by one solve with $(\mathbf{A}^\top-\sigma \mathbf{I})$, while the residual $\|\mathbf{A}^\top \mathbf{y}-\lambda \mathbf{y}\|_2$ provides a reliable stopping criterion.

% In the forensic implementation, the same per-shift routine is applied across all $\sigma_\ell\in\Sigma$ and every successful returned certificate is retained; since there is at most one accepted certificate per shift, this stage produces at most $q_\sigma$ certificate vectors $\{\mathbf y_{\sigma_\ell}\}_{\ell=1}^{q_\sigma}$.

		The resulting iterate is then tested for certificate validity by checking \textit{(i)} near-orthogonality to the input directions $\|\mathbf{B}^\top \mathbf{y}\|_\infty \leq \epsilon_\mathbf{B}$, and \textit{(ii)} target violation, $|\mathbf{y}^\top\Delta \mathbf{x}|$ being non-negligible. \textcolor{black}{The condition $\|\mathbf{B}^\top \mathbf{y}\|_\infty \leq \epsilon_\mathbf{B}$ is a \emph{numerical proxy} for the theoretical condition $\boldsymbol{\mathcal{C}}_N^\top \mathbf{y} = \mathbf{0}$.  Since Alg.~\ref{alg:inviter} enforces $\mathbf{y}$ to be an eigenvector (or generalized eigenvector) of $\mathbf{A}^\top$, the orthogonality to the entire controllability matrix $\boldsymbol{\mathcal{C}}_N$ is guaranteed if $\mathbf{y}$ is orthogonal to the first block $\mathbf{B}$. This is because $\mathbf{y}^\top \mathbf{A}^k \mathbf{B} = (\mathbf{y}^\top \mathbf{A}^k) \mathbf{B} = (\lambda^k \mathbf{y}^\top) \mathbf{B} = \lambda^k (\mathbf{y}^\top \mathbf{B})$. Thus, verifying $\mathbf{y}^\top \mathbf{B} \approx \mathbf{0}$ is computationally sufficient and allows us to avoid the prohibitive memory and time costs of forming $\boldsymbol{\mathcal{C}}_N$ explicitly. This proxy is exact when $\mathbf{y}$ is a true eigenvector and approximate to within $\|\mathbf{A}^\top\mathbf{y}-\lambda\mathbf{y}\|_2 \leq \varepsilon_{\mathrm{eig}}$ otherwise. }

We also note the following. \textit{(i)} The small but nonzero values of \(\|\mathbf{B}^\top \mathbf{y}\|_\infty\) may also indicate weakly actuated directions, and should therefore be interpreted as numerical evidence of near-uncontrollability rather than exact algebraic uncontrollability. \textit{(ii)} Alg.~\ref{alg:inviter} does not form the inverse of \((\mathbf A^\top-\sigma \mathbf I)\) explicitly. For each fixed shift, the method repeatedly solves linear systems with the same coefficient matrix, so a single sparse or dense factorization can be reused throughout the $K_{\max}$ inverse-iteration updates at that shift. Thus, the dominant dense cost is one \(O(n^3)\) LU factorization per shift, followed by \(O(n^2)\) backsolves per iteration. In sparse settings, the cost is governed by the corresponding sparse linear solver and can scale much more favorably than a global eigendecomposition. \textit{(iii)} The distinction between \emph{certification mode} and \emph{forensic mode} is important: certification stops at the first successful certificate, while the forensic scan intentionally continues through the whole shift set in order to collect up to one accepted certificate per shift before spectral extraction. \textit{(iv)} Algorithm~\ref{alg:inviter} is a one-sided certificate-search routine. A successful return certifies infeasibility of the prescribed target displacement, whereas failure after scanning the finite shift set \(\Sigma\) is inconclusive and should not be interpreted as a proof of feasibility.

\normalcolor

\begin{algorithm}[t]
	\caption{Inverse iteration for certificate $\mathbf{y}$}
	\label{alg:inviter}\label{alg:cert}
	\footnotesize
	\DontPrintSemicolon
	\setlength{\algomargin}{0.5em}
	\SetAlgoSkip{smallskip}
	\SetInd{0.2em}{0.4em}
	\SetKwInOut{Input}{Input}\SetKwInOut{Output}{Output}
	\Input{$\mathbf{A},\mathbf{B},\Delta \mathbf{x}$, $\Sigma=\{\sigma_\ell\}_{\ell=1}^{q_\sigma}$, $\varepsilon_{\mathrm{eig}},\varepsilon_\mathbf{B}$, $\tau_\Delta$, $K_{\max}$}
\textcolor{black}{	\Output{$(\sigma^\star,\mathbf{y}_{\sigma^\star})$ or failure}}
	\For{$\ell=1,\dots,q_\sigma$}{
		\textcolor{black}{$\sigma\leftarrow \sigma_\ell$,
		$\mathbf{y}_{\sigma}\leftarrow$ random, $\mathbf{y}_{\sigma}\leftarrow \mathbf{y}_{\sigma}/\|\mathbf{y}_{\sigma}\|_2$}, $\mathbf{M}\leftarrow (\mathbf{A}^\top-\sigma \mathbf{I})$\;
		\For{$k=1,\dots,K_{\max}$}{
			Solve $\mathbf{M}\mathbf{z}=\mathbf{y}_{\sigma}$, $\mathbf{y}_{\sigma}\leftarrow \mathbf{z}/\|\mathbf{z}\|_2$, $\lambda\leftarrow \frac{\mathbf{y}_{\sigma}^\top \mathbf{A}^\top \mathbf{y}_{\sigma}}{\mathbf{y}_{\sigma}^\top \mathbf{y}_{\sigma}}$\;
			\If{$\|\mathbf{A}^\top \mathbf{y}_{\sigma}-\lambda \mathbf{y}_{\sigma}\|_2\le \varepsilon_{\mathrm{eig}}$}{break\;}
		}
		\If{$\|\mathbf{B}^\top \mathbf{y}_{\sigma}\|_\infty\le\varepsilon_\mathbf{B}$ and $|\mathbf{y}_{\sigma}^\top\Delta \mathbf{x}|>\tau_\Delta$}{\KwRet{$(\sigma,\mathbf{y}_{\sigma})$}\;}
	}
%	\KwRet{\textcolor{black}{failure after scanning $\Sigma$}}\;
\end{algorithm}

%{\small\auditnote{Important implementation detail: the code does \emph{not} require the final single-vector eigen-residual to satisfy $\varepsilon_{\mathrm{eig}}$ before declaring success. That is not necessarily a problem if the accepted object is a \emph{mixed} certificate spanning several uncontrollable modes, but then the manuscript should describe this accurately rather than implying that every accepted $\mathbf y$ is a numerically converged single eigenvector.}}
		
		\vspace{0.15cm}
		\noindent \textsc{\textbf{Novelty}.} The structural results in Theorem~\ref{thm:PBHcert} and Corollary~\ref{thm:PBHcertnond} unify two distinct lineages of control theory: the geometric approach rooted in convex duality~\cite[Section 5.8]{boyd2004convex} and the spectral approach grounded in modal analysis in~\cite[Section 6.2]{kailath1980linear} and \cite{hautus1969controllability}. While the classical PBH test attributes uncontrollability to specific eigenmodes and the Fredholm Alternative characterizes infeasibility via algebraic certificates, our results provide a strong connection between these domains. The novelty of this approach lies in its ability to bypass the global spectrum. While the traditional PBH test does not take the target $\Delta \mathbf{x}$ into account, our implementation exploits the dual certificate to identify only the modes that obstruct the desired maneuver. 
		
%		By framing the certificate search through shifted inverse iteration, the algorithm becomes inherently robust to the system's underlying structure. As established in Corollary~\ref{thm:PBHcertnond}, for non-diagonalizable systems, the iteration naturally converges toward the uncontrollable generalized eigenspace. This enables a unified computational method without prior knowledge of whether the system possesses Jordan blocks.
		
		Despite extensive literature on controllability and convex duality, these topics have largely been treated in isolation. \textcolor{black}{To the best of our knowledge, no prior work has provided an explicit computational connection between the geometric and spectral characterizations of uncontrollability. While connections between controllability and duality appear in various forms in the literature (e.g.,~\cite{balakrishnan2003sdp} that links SDP duality to Lyapunov/Riccati equations) the present framework uses dual certificates as the central objects for directly recovering uncontrollable PBH modes.} Specifically, this framework enables the study of PBH uncontrollability through infeasibility certificates by \textit{(i)} treating a certificate direction as the primary object, \textit{(ii)} using it to generate a minimal invariant Krylov subspace, and \textit{(iii)} extracting the participating uncontrollable modes without computing a global eigendecomposition.

\begin{algorithm}[t]
	\caption{Spectral extraction from certificates}
	\label{alg:extractionKry}\label{alg:extraction}
	\footnotesize
	\DontPrintSemicolon
	\setlength{\algomargin}{0.5em}
	\SetAlgoSkip{smallskip}
	\SetInd{0.2em}{0.4em}
	\SetKwInOut{Input}{Input}\SetKwInOut{Output}{Output}
	\Input{$\mathbf{A}$, $\{\mathbf{y}_{\sigma_\ell}\}_{\ell=1}^{q_\sigma}$,  $r_{\max}$, $\varepsilon_{\mathrm{rank}},\varepsilon_{\mathrm{ls}}$, 	$\Lambda\leftarrow \emptyset$}
\textcolor{black}{	\Output{Extracted eigenvalues $\{\lambda_j\}$}}
	\For{$\ell=1,\dots,q_\sigma$}{
		\If{$\mathbf{y}_{\sigma_\ell}$ is not successful}{continue}
		$\mathbf{k}_0\leftarrow \mathbf{y}_{\sigma_\ell}$, $\mathbf{K}\leftarrow[\mathbf{k}_0]$, $r\leftarrow 0$\;
		\For{$j=1,\dots,r_{\max}$}{
			$\mathbf{k}_j\leftarrow \mathbf{A}^\top \mathbf{k}_{j-1}$, $\mathbf{K}\leftarrow[\mathbf{K},\mathbf{k}_j]$\;
			\If{$\mathrm{rank}_\varepsilon(\mathbf{K})=\mathrm{rank}_\varepsilon(\mathbf{K}(:,1\!:\!j))$}{$r\leftarrow j$, \KwBreak}
		}
		\If{$r=0$}{continue}
		$\mathbf{K}_0\leftarrow \mathbf{K}(:,1\!:\!r)$, $\mathbf{K}_r\leftarrow \mathbf{K}(:,r\!+\!1)$, $\mathbf{c}\leftarrow \arg\min_{\mathbf{c}}\|\mathbf{K}_0\mathbf{c}-\mathbf{K}_r\|_2$\;
		\If{$\|\mathbf{K}_0\mathbf{c}-\mathbf{K}_r\|_2\le\varepsilon_{\mathrm{ls}}$}{
			$P(z)\leftarrow z^r-\sum_{i=0}^{r-1}c_{i+1}z^i$\;
		\textcolor{black}{	$\Lambda\leftarrow \Lambda \cup \{\text{roots of }P\}$\;}
		}
	}
%\textcolor{black}{\KwRet{numerically merged $\Lambda$}\;}
\end{algorithm}

		\section{Spectral Extraction from Certificates}\label{sec:extraction}

		While Theorem~\ref{thm:PBHcert} and Corollary~\ref{thm:PBHcertnond} establish that $\mathbf{y}$ is a linear combination of uncontrollable modes, a practical challenge remains: identifying which specific eigenvalues $\lambda_i$ are associated with a given certificate $\mathbf{y}$. 
		\begin{theorem}[Extraction of Uncontrollable Modes]\label{thm:extraction}
			{Let $\mathbf{y} \in \mathbb{R}^n$ be a valid uncontrollability certificate. Let $\mathcal{K}_r(\mathbf{A}^\top, \mathbf{y}) = \mathrm{span}\{\mathbf{y}, \mathbf{A}^\top \mathbf{y}, \dots, (\mathbf{A}^\top)^{r-1} \mathbf{y}\}$ be the $\mathbf{A}^\top$-invariant Krylov subspace generated by $\mathbf{y}$, where $r$ is the minimal dimension such that $\mathbf{A}^\top \mathcal{K}_r \subseteq \mathcal{K}_r$, and $\bm{\mathcal{I}}_\mathbf{y}\subseteq\bm{\mathcal{I}}$ indexes the uncontrollable eigenvalues whose generalized eigenspace contains vectors appearing with a nonzero coefficient in the expression of $\mathbf{y}$ as linear combination of generalized eigenvectors. Then the set of eigenvalues of the restricted operator $\mathbf{A}^\top|_{\mathcal{K}_r}$ is exactly the subset of uncontrollable eigenvalues $\{\lambda_j\}_{j \in \bm{\mathcal{I}}_\mathbf{y}}$ that contribute to the certificate $\mathbf{y}$.}
		\end{theorem}
		
		\begin{proof}
			{By Theorem~\ref{thm:PBHcert}, any certificate $\mathbf{y}$ can be expressed as $\mathbf{y} = \sum_{j \in \bm{\mathcal{I}}_\mathbf{y}} \alpha_j \mathbf{v}_j$, where each $\mathbf{v}_j$ denotes the full projection of $\mathbf{y}$ onto the generalized eigenspace of $\lambda_j$.
				Since the subspaces spanned by generalized eigenvectors associated with distinct eigenvalues are linearly independent, the Krylov subspace $\mathcal{K}_r$ is the direct sum of the cyclic subspaces generated by each $\mathbf{v}_j$. The minimal polynomial of $\mathbf{A}^\top$ restricted to $\mathcal{K}_r$ is $P(z) = \prod_{j \in \bm{\mathcal{I}}_\mathbf{y}} (z - \lambda_j)^{m_j}$, where $m_j$ is the size of the largest Jordan block associated with $\lambda_j$ in the minimal polynomial. Consequently, the eigenvalues of the reduced-order operator are precisely the uncontrollable eigenvalues composing $\mathbf{y}$.}
		\end{proof}
		
		Theorem~\ref{thm:extraction} shows that, instead of performing a full $O(n^3)$ spectral decomposition of $\mathbf{A}$, one can identify the uncontrollable modes by solving a small-scale least-squares problem to find the coefficients of the minimal polynomial associated with $\mathbf{y}$. If $\mathbf{y}$ is composed of $r$ modes, where $r \ll n$, the uncontrollable eigenvalues are found by finding the roots of a polynomial of degree $r$. Once $\lambda_j$ is identified, the associated uncontrollable eigenvector $\mathbf{w}_j$ can be recovered via the kernel of $\mathbf{A}^\top - \lambda_j \mathbf{I}$.

\noindent\textsc{\textbf{Implementation and Novelty}.} Alg.~\ref{alg:extractionKry} implements Theorem~\ref{thm:extraction} by converting any infeasibility certificate $\mathbf{y}$ into a small spectral problem. In the present framework this certificate is specifically the output $\mathbf{y}_{\sigma^\star}$ returned by Alg.~\ref{alg:inviter} for some accepted shift $\sigma^\star\in\Sigma$. \textcolor{black}{In the spectral extraction, the same idea is applied to each successful certificate returned across the scanned shifts, and the resulting eigenvalues are merged numerically afterward.} The approach relies on the fact that if $\mathbf{y}_{\sigma^\star}$ is composed of $r$ uncontrollable modes, the set of vectors $\{\mathbf{y}_{\sigma^\star}, \mathbf{A}^\top \mathbf{y}_{\sigma^\star}, \dots, (\mathbf{A}^\top)^r \mathbf{y}_{\sigma^\star}\}$ must be linearly dependent. The algorithm builds a Krylov sequence using only matrix-vector products, detects the minimal order $r$ via rank saturation of the resulting Krylov matrix $\mathbf{K}$, and fits a linear relation to recover the minimal polynomial $P(z)$ associated with $\mathbf{y}_{\sigma^\star}$. The roots of $P$ are exactly the uncontrollable eigenvalues participating in $\mathbf{y}_{\sigma^\star}$ (Theorem~\ref{thm:extraction}). \textcolor{black}{In the current implementation, the least-squares residual is also used as a confidence check: small residuals are labeled as accepted extractions, while larger residuals are retained but marked as low-confidence rather than discarded.} This procedure does not merely replace one eigenvalue computation with another: the PBH workflow requires resolving eigenstructure in $\mathbb{R}^n$, whereas the proposed extraction computes eigenvalues of a companion matrix of size $r\times r$ with $r\ll n$. Hence, the spectral computation is reduced from an $n$-dimensional task to a low-order algebraic operation.

		\noindent	\textcolor{black}{\textsc{\textbf{Certification-Extraction (C+E) Integration.}}	Algorithms~\ref{alg:inviter} and~\ref{alg:extractionKry} play sequential and complementary roles in the proposed certificate-and-extraction framework. In certification mode, Alg.~\ref{alg:inviter} is used exactly as written: it scans $\Sigma=\{\sigma_\ell\}_{\ell=1}^{q_\sigma}$ and returns the first successful pair \((\sigma^\star,\mathbf{y}_{\sigma^\star})\) satisfying the infeasibility conditions. In forensic mode, the same per-shift certificate routine is evaluated across all shifts, yielding a set of at most $q_\sigma$ successful certificates \(\{\mathbf y_{\sigma_\ell}\}\), because there is at most one accepted certificate associated with each shift. Alg.~\ref{alg:extractionKry} is then the spectral extraction step: it is applied to each successful certificate, and the resulting extracted eigenvalues are merged numerically afterward. If a certificate is obtained directly from an optimization solver that solves the MECP~\eqref{equ:mecp_alt}, Alg.~\ref{alg:inviter} is bypassed and Alg.~\ref{alg:extractionKry} is applied directly to that solver-generated certificate.}

		\begin{remark}[Convergence Guarantees]\label{rem:convergence} Alg.~\ref{alg:inviter} is shifted inverse iteration on $\mathbf{A}^\top$. Under the assumption that $\sigma$ is not an eigenvalue of $\mathbf{A}$ and that the initial vector $\mathbf{y}^{(0)}$ has a nonzero component along the target eigenvector $\mathbf{w}^*$ (which holds with probability $1$ for a random start), the iterates converge to $\mathbf{w}^*$ at linear rate; see~\cite[Sec. 4.1.3]{saad2011numerical} for more discussions on this.  Alg.~\ref{alg:extractionKry} terminates finitely as rank saturation of the Krylov matrix at order $r$ is guaranteed to occur at $r = \dim\,\mathcal{K}_r(\mathbf{A}^\top,\mathbf{y})\leq n_{\mathrm{unc}}$, which is the number of distinct uncontrollable eigenvalues in $\mathbf{y}$.
		\end{remark}

        \begin{remark}[Complex Eigenvalues]\label{rem:complex}\
			The developments extend naturally to systems with complex eigenvalues. The PBH condition is stated with the conjugate transpose and reads $\mathbf{w}^{H}\mathbf{A}=\lambda\,\mathbf{w}^{H}$ together with $\mathbf{w}^{H}\mathbf{B}=\mathbf{0}$, and the infeasibility condition in \eqref{equ:certificate_alt} is written as $\boldsymbol{\mathcal{C}}_{N}^{H}\mathbf{y}=\mathbf{0}$ and $\mathbf{y}^{H}\Delta\mathbf{x}\neq 0$; when targeting complex modes directly, Alg.~\ref{alg:inviter} extends by taking $\sigma\in\mathbb{C}$ and replacing the Rayleigh quotient with its Hermitian form $\lambda\leftarrow\mathbf{y}^H\mathbf{A}^\top\mathbf{y}/\mathbf{y}^H\mathbf{y}$. For real data $(\mathbf{A},\mathbf{B},\Delta\mathbf{x})$, solvers typically return real certificates; in that case, uncontrollable complex conjugate pairs induce real invariant subspaces spanned by $\Re(\mathbf{w})$ and $\Im(\mathbf{w})$, and the certificate decomposition implied by Theorem~\ref{thm:PBHcert} and Corollary~\ref{cor:interpretuncont} remains valid when interpreted over these real invariant subspaces. Consequently, the extraction mechanism in Theorem~\ref{thm:extraction} and Alg.~\ref{alg:extractionKry} recovers oscillatory uncontrollable modes as conjugate-paired eigenvalues when the reduced polynomial (or reduced operator) is formed with real coefficients.
		\end{remark}
		\textcolor{black}{We finally note that the present letter is focused on discrete-time systems. While the same certificate viewpoint extends naturally to continuous time after replacing the finite controllability matrix by the continuous-time reachability operator, this extension is outside the scope of this work.}
		
		\section{Case Studies}\label{sec:case}

        \color{black}
This section investigates the practicality of the developed algorithms. All simulations are performed on a 2026 Macbook Pro M5 Max with 128 GB of RAM; we used Python for all tests. We validate the methods in this letter on a generic routing/traffic model $\mathbf{x}_{k+1}=\mathbf{A}\mathbf{x}_k+\mathbf{B}\mathbf{u}_k$, where $\mathbf{x}_k$ stacks link occupancies at time step $k$ and $\mathbf{u}_k$ represents actuated inflows. The model is inspired by cell-transmission and \textit{store-and-forward} traffic dynamics in which a routing matrix encodes turning fractions and link-to-link propagation of traffic mass~\cite{aboudolas2009store}. In the benchmark generator, $\mathbf{A}$ is partitioned as $\mathbf{A}=\mathrm{blkdiag}(\mathbf{A}_{\mathrm{ctrl}},\mathbf{A}_{\mathrm{unc}})$ and $\mathbf{B}=[\,\mathbf{B}_{\mathrm{ctrl}}^\top\ \mathbf{0}^\top\,]^\top$, with $\mathbf{A}_{\mathrm{ctrl}}=\gamma\big((1-\alpha)\mathbf{I}+\alpha\mathbf{P}\big)$. We fix $\alpha=0.25$, $\gamma=0.98$, $m=\lfloor 0.1n\rfloor$, and horizon $N=n$. We test dense and sparse network structures. In the dense topology, $\mathbf{P}$ is a dense row-stochastic routing matrix with added self-weight; in the sparse one, $\mathbf{P}$ is a ring-type routing operator built from nominal forward/backward/self fractions $0.80/0.15/0.05$ before normalization. The block $\mathbf{A}_{\mathrm{unc}}$ plants real uncontrollable eigenvalues in $[-2,2]$ (unknown to our algorithms), and the displacement $\Delta\mathbf{x}$ is forced to excite those coordinates so that the benchmark contains prescribed unreachable maneuvers. 	We compare the following  methods:

			% \mknote{Space saver: replace the itemize below with Table~\ref{tab:methods} right after it. Saves $\sim$6--8 lines. Delete whichever version you don't want.}
			\begin{itemize}
			\item {\textbf{C}}: Alg.~\ref{alg:inviter} is used only for uncontrollability certification; it stops at the first successful infeasibility certificate. 
			\item {\textbf{C+E}}: the full certificate-and-extraction approach; it scans all prescribed $\sigma_l \in \Sigma$, applies Alg.~\ref{alg:extractionKry} to every successful certificate, and merges duplicate extracted eigenvalues.
			\item {\textbf{PBH}}: the classical left-eigenvector PBH test based on the full eigendecomposition of $\mathbf A$, via SciPy's \texttt{LAPACK}.
			\item \textbf{Gram}: the exact finite-horizon controllability Gramian, computed by the efficient power-doubling recursion~\cite{smith1968matrix}.
			\item \textbf{D-LS}: dense least squares to solve $||\boldsymbol{\mathcal{C}}_N \mathbf{u} - \Delta \mathbf{x} ||$ with some efficiency improvements.
			\item \textbf{MF-LS}: matrix-free least squares applied to the same problem without assembling $\boldsymbol{\mathcal{C}}_N$. 
		\end{itemize}

% \begin{table}[h]
% \centering
% \footnotesize
% \setlength{\tabcolsep}{3pt}
% \renewcommand{\arraystretch}{1.05}
% \caption{ Methods compared in the case studies. \mknote{Need to adjust the descriptions}}
% \label{tab:methods}
% \begin{tabular}{@{}p{0.12\linewidth}p{0.84\linewidth}@{}}
% \toprule
% \textbf{Method} & \textbf{Description} \\
% \midrule
% \textbf{C}     & Alg.~\ref{alg:inviter} for certification only; stops at the first successful certificate. \\
% \textbf{C+E}   & Full certificate-and-extraction: scans all $\sigma_\ell\!\in\!\Sigma$, applies Alg.~\ref{alg:extractionKry} to every successful certificate, merges duplicates. \\
% \textbf{PBH}   & Classical left-eigenvector PBH test via full eigendecomposition (SciPy/\texttt{LAPACK}). \\
% \textbf{Gram}  & Finite-horizon controllability Gramian via power-doubling~\cite{smith1968matrix}. \\
% \textbf{D-LS}  & Dense least squares on $\|\boldsymbol{\mathcal{C}}_N\mathbf{u}-\Delta\mathbf{x}\|$. \\
% \textbf{MF-LS} & Matrix-free least squares; $\boldsymbol{\mathcal{C}}_N$ never assembled. \\
% \bottomrule
% \end{tabular}
% \end{table}

		\begin{figure}[t]
%			\centering
	% \vspace{-0.3cm}
\hspace{-0.4cm}			\includegraphics[scale=0.95]{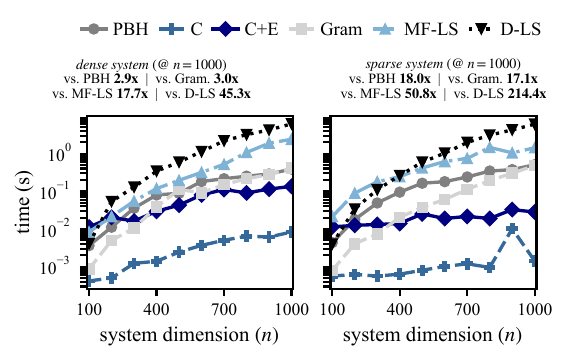}
	\vspace{-0.8cm}
			\caption{ Compute time for all methods (left (right) is for dense (sparse) systems). Ratio of improvement is with respect to \textbf{C+E}. }
			\label{fig:scaling_comparison}
            \vspace{0.4cm}
		\end{figure}

\begin{table}[t]
\centering
\footnotesize
\setlength{\tabcolsep}{2pt}
\renewcommand{\arraystretch}{0.92}
\caption{ Benchmark parameters used in the case studies.}
\label{tab:case_params}
\begin{tabular}{p{0.96\linewidth}}
\toprule
$q_\sigma=10$, $K_{\max}=20$,  $\{\varepsilon_B, \varepsilon_{\mathrm{eig}}, \tau_\Delta, \varepsilon_{\mathrm{PBH}}\}=10^{-10}$, $r_{\max}=50$, $\varepsilon_{\mathrm{rank}}=10^{-14}$, $\varepsilon_{\mathrm{ls}}=10^{-6}$, $\varepsilon_{\mathrm{ext}}=10^{-3}$. \\
\bottomrule
\end{tabular}
\end{table}

\begin{figure}[t]
	\centering
    % \vspace{-0.2cm}
		\includegraphics[scale=0.92]{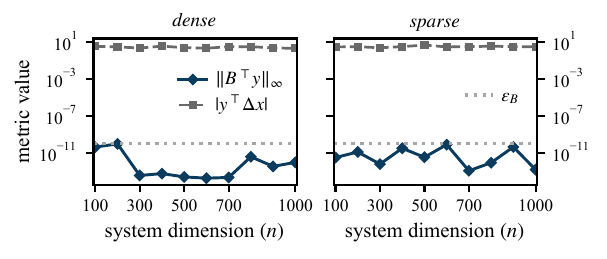}\vspace{-0.2cm}
		\caption{ Alg.~\ref{alg:inviter} certificate diagnostics: the infeasibility constraints are satisfied for both types of networks.}
		\label{fig:infeas}
        \vspace{0.4cm}
\end{figure}

\begin{figure}[t]
	\centering
		\includegraphics[scale=0.93]{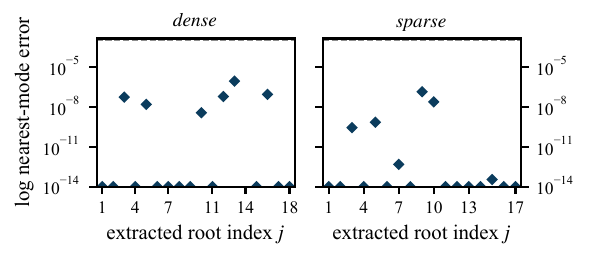} 

% \vspace*{-0.25cm}	
        
		\includegraphics[scale=0.95]{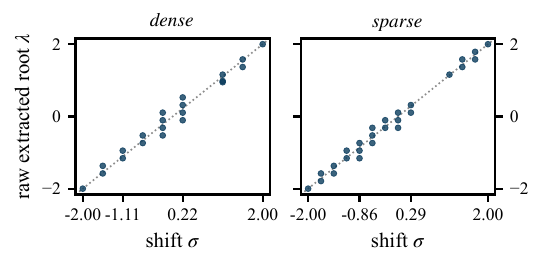}\vspace{-0.2cm}
        
		\caption{(\textit{Top}) Nearest-mode extraction error for \textbf{C+E} at $n=200$ with $n_{\mathrm{unc}}=20$ uncont. modes.  \textit{(Bottom): Extracted roots for every shift $\sigma_l$.}}
        \label{fig:extract}
        \vspace{0.3cm}
	\end{figure}

Tab.~\ref{tab:case_params} lists all relevant parameters.  Fig.~\ref{fig:scaling_comparison} shows the computational time for all methods. It demonstrates that the certification only via \textbf{C} is an efficient way to certify uncontrollability. When combined   with the spectral extraction, \textbf{C+E} remains competitive with the classical diagnostics. At the final benchmark size $n=1000$, \textbf{C+E} is about three times faster than \textbf{PBH} and the  \textbf{Gram} in the dense case, and about $18$ and $17$ times faster than \textbf{PBH} and the \textbf{Gram} in the sparse case. The least-squares based approaches are the slowest (no tuning of \textbf{MF-LS} produced major improvements), with the dense one predictably being the slowest.  The advantage of the proposed C+E approach becomes much larger in sparse systems. 

Fig.~\ref{fig:infeas} confirms the proxies of the certificate conditions numerically with $\|\mathbf{B}^\top \mathbf{y}\|_\infty$ remaining below the threshold, while $|\mathbf{y}^\top \Delta \mathbf{x}|$ stays strictly positive and of order one (thereby confirming~\eqref{equ:certificate_alt}). The spectral extraction shows what the full \textbf{C+E} approach recovers from those certificates in Fig.~\ref{fig:extract} (top) via the nearest-mode error. In the dense case, the extraction returns $18$ extracted eigenvalues from $9$ successful certificates and a total of $n_{\mathrm{unc}} = 20$ modes; in the sparse case, it returns $17$ extracted eigenvalues from $13$ successful certificates ($n_{\mathrm{unc}} = 20$).  Fig.~\ref{fig:extract} provides the complementary raw view---the extracted roots track the scanned shifts across the planted uncontrollable interval, which is consistent with the intended role of the shift scan as a coarse spectral coverage rather than as a precise eigenvalue estimator.
		\normalcolor

		% \mknote{Missing continuous-time remark: the response to R4.8b promises ``The framework extends to continuous-time systems by replacing the discrete-time controllability matrix $\mathcal{C}_N$ with the continuous Gramian; the PBH test and the certificate conditions are identical in both settings.'' This text does not appear anywhere in the manuscript. Either add it (e.g.\ as a one-sentence remark or footnote) or correct the response document. A reviewer who checks will notice this.}
	\section{Letter Summary and Future Work}\label{sec:summ}

	The practical message of this letter is threefold: \textit{(i)} a procedure
has been given to efficiently certify that a prescribed terminal state
$\Delta \mathbf{x}$ is unreachable, without forming the controllability matrix or performing a global eigendecomposition; \textit{(ii)} the resulting certificate \(\mathbf{y}\) can be passed directly to recover the specific uncontrollable modes obstructing that maneuver; and \textit{(iii)} the proposed methods are computationally inexpensive, making them amenable to  large-scale systems with thousands of states. This makes the methods even more relevant when the system matrices are recomputed for various linearization points.
		This work has limitations. Alg.~\ref{alg:extractionKry} is most reliable when the certificate contains a modest number of contributing modes (on the order of a few tens). As the extraction order $r$ increases, the Krylov matrix can become ill-conditioned, and the subsequent root-finding step becomes sensitive to finite precision perturbations. This is a standard numerical phenomenon: small coefficient errors in a degree-$r$ polynomial may induce disproportionately larger perturbations in its roots, particularly when roots are clustered or poorly separated. In our case studies, we observe accurate modal recovery for moderate $r$ (e.g., $r\approx 20$--$30$ of uncontrollable modes), while larger values of $r$ may exhibit degraded accuracy; thereby motivating more numerically stable reduced-order spectral realizations as a direction for future work.

\normalcolor

{
		\bibliographystyle{IEEEtran}
		\bibliography{letter_refs}}
		
	\end{document}